\documentclass[11pt]{article}

\usepackage{amsmath,amssymb,amsthm}
\usepackage{hyperref}
\usepackage{geometry}
\usepackage{xcolor}

\usepackage{ulem}
\usepackage{cancel}
\title{Ramification Subgroups of Knot Groups and their Profinite and Cohomological Structure}
\author{Marina Palaisti \& Federico W. Pasini}
\date{\today}

\newtheorem{theorem}{Theorem}[section]
\newtheorem{proposition}[theorem]{Proposition}
\newtheorem{lemma}[theorem]{Lemma}
\newtheorem{definition}[theorem]{Definition}
\newtheorem{corollary}[theorem]{Corollary}
\newtheorem{remark}[theorem]{Remark}

\newtheorem{hyp}{Standing Hypothesis}

\begin{document}
	\maketitle
	
	\begin{abstract}
		We formalize a ramification theory for finite covers of knot exteriors. Given a knot group $G_K$ and a finite-index subgroup $U\le G_K$, we define meridional inertia subgroups $U\cap g\langle m\rangle g^{-1}$ and the global ramification subgroup $M_U\triangleleft U$ as their normal closure. We then analyze $M_U$ from three complementary viewpoints:
		(1) finite quotients, where $U/M_U$ is shown to be the universal ``maximal meridionally unramified'' quotient of $U$;
		(2) profinite completions, where we identify the closed ramification subgroup $\widehat M_{\widehat U}$ as the closed normal subgroup generated by closed inertia and prove that meridian-preserving isomorphisms of profinite completions preserve inertia and ramification;
		(3) cohomology, where ``unramified'' $H^1$-classes (discrete and profinite) are characterized as those vanishing on all inertia subgroups, in direct analogy with number-theoretic inertia conditions in Galois cohomology.
	\end{abstract}
	
	\section{Introduction}
	
	Arithmetic topology proposes a dictionary between knots and primes, with branched coverings corresponding to ramified extensions and peripheral/meridional behaviour corresponding to inertia (cf.\ Morishita \cite{Morishita}). In this paper we isolate, for arbitrary finite covers of knot exteriors, a canonical normal subgroup generated by lifted meridians and study it as a ramification subgroup.
	
	Let $K\subset S^3$ be a tame knot, with tubular neighbourhood $\nu(K)$ and with exterior
	$
	X_K := S^3\setminus \operatorname{int}(\nu(K))$; let further $G_K := \pi_1(X_K, x_0)$ for any basepoint $x_0$, and fix a meridian $m\in G_K$.
	
	Given a finite-index subgroup $U\le G_K$, the corresponding cover $X_U\to X_K$ is a 3-manifold with boundary given by a disjoint union of tori, and each boundary component has a distinguished meridian slope (a lift of $m$). The ordinary way to annihilate those lifted meridians is exactly the group-theoretic operation that occurs under Dehn filling and under formation of branched covers in the cyclic case. In this article, we package all such meridian lifts into a single \textit{ramification subgroup} $M_U\triangleleft U$ and study its behaviour under finite quotients, profinite completions, and cohomology, in order to highlight the analogy between $M_U$ and number-theoretic inertia subgroups.
	
	Our point of view is complementary to existing arithmetic-topological and profinite-rigidity work on knot groups. The \textit{ambient classifying spaces for knot groups} \cite{PasiniClassifying, PasiniAmbClassSp} organize meridians as a distinguished family in order to build spaces reflecting ``unramified'' and ``ramified'' behaviour in an algebraic-number-theoretic sense, shows that meridian data are central in rigidity questions. The present paper keeps the setting at the level of (profinite) groups and $H^1$, and isolates a concrete subgroup $M_U$ that plays the role of an \textit{inertia-killing} subgroup for finite-index subgroups $U\le G_K$.
	
	\subsection*{Ramification dictionary}
	
	The constructions introduced in this paper are motivated by the classical analogy between number theory and knot theory:
	
	\begin{center}
		\begin{tabular}{c|c}
			Number theory & Knot theory \\[0.2em]
			\hline\\[-0.8em]
			number field $K$ & knot complement $X_K$ \\
			Galois group $G_K$ & knot group $\pi_1(X_K)$ \\
			algebraic extension $L$ over $K$ & finite connected cover $X_U\to X_K$\\
			prime ideal $\mathfrak p$ & meridian $m$ \\
			inertia subgroup $I_{\mathfrak p}$ & meridional inertia $U\cap g\langle m\rangle g^{-1}$ \\
			ramification subgroup & $M_U$ \\
			maximal unramified extension & quotient $U/M_U$
		\end{tabular}
	\end{center}
	
	This dictionary is inspired by the framework of arithmetic topology developed by Morishita \cite{Morishita}. The ramification subgroup fits into the short exact sequence
	
	\[\begin{array}{ccccccccc}
		1 & \to & M_U & \to & U & \to & U/M_U & \to & 1 \\
		& & \downarrow & & \downarrow & & \downarrow & & \\
		1 & \to & \overline{M_U} & \to & \widehat U & \to & \widehat U/\overline{M_U} & \to & 1
	\end{array}\]
	
	The lower row describes the corresponding quotient in the profinite completion. One of the main results of the paper shows that the closed ramification subgroup $\overline{M_U}$ is generated by the closures of the meridional inertia subgroups.
	
	\medskip
	\noindent\textbf{Organization.}
	Section~2 recalls peripheral structure and introduces meridional inertia and the ramification subgroup. Section~3 studies the maximal meridionally unramified quotient $U/M_U$, including its universal property, its behaviour in finite quotients, and the discrete $H^1$ characterization of unramified classes. Section~4 develops the profinite ramification subgroup via closures of inertia and ramification. Section~5 proves that meridian-preserving profinite isomorphisms preserve closed inertia and closed ramification. Section~6 gives the profinite cohomological characterization of unramified classes.
	
	\section{Peripheral structure, meridional inertia, and the ramification subgroup}
	Knots in this paper are tacitly assumed to be \textit{tame} (cf.~\cite[Section 1.A]{burzie}, \cite[Section 1.2]{crofox}). This is a natural restriction in the study of knots from the perspective of algebraic topology, but in our case it is essential, as \textit{wild} (that is, not tame) knots do not have a tubular neighbourhood.
	Fix a nontivial knot $K\subset S^3$, let $\nu(K)$ be its tubular neighbourhood, and let $X_K = S^3\setminus\operatorname{int}(\nu(K))$ be its (closed) knot exterior: then $\partial X_K$ is a torus. Fix a basepoint $x_0\in \partial X_K$, and let $G = G_K = \pi_1(X_K,x_0)$.
	
	\begin{definition}[Meridian and preferred longitude]
		A \emph{meridian} $m\in \pi_1(\partial X_K,x_0)$ is represented by a simple closed curve on $\partial X_K$ bounding an embedded disk in $\nu(K)$ that meets $K$ once. A \emph{preferred longitude} $\ell\in \pi_1(\partial X_K,x_0)$ is represented by a simple closed curve on $\partial X_K$ that is null-homologous in $X_K$ and such that $\{m,\ell\}$ is a basis for $H_1(\partial X_K)\cong \mathbb Z^2$.
	\end{definition}
	Intuitively, the meridian is [the equivalence class of] a curve that starts at the basepoint $x_0$, loops once around the knot, and returns to $x_0$ doing nothing else interesting; the longitude is [the equivalence class of] a curve that starts at the basepoint $x_0$, runs once along the knot, and returns to $x_0$ doing nothing else interesting.
	
	Let $i:\partial X_K\hookrightarrow X_K$ be inclusion.
	
	\begin{definition}[Peripheral subgroup]
		The \emph{peripheral subgroup} is
		\[ P_K := i_*\big(\pi_1(\partial X_K,x_0)\big)\le G.\]
		By abuse of notation we identify $m,\ell$ with their images in $G$. Then
		\[P_K=\langle m,\ell\rangle \cong \mathbb Z^2.\]
	\end{definition}
	
	As a well known consequence of Alexander duality, the abelianization of $G$ is
	$G^{ab}\cong H_1(X_K)\cong \mathbb Z$, generated by the meridian class. In particular, the meridian $m$ is unique up to conjugation, which guarantees all the following concepts to be well-defined and essentially independent of the choice of $m$. We can therefore make the following
	
	\begin{hyp}
		Throughout the paper, let $m\in G$ be the chosen meridian, let $U\le G$ be a finite-index subgroup, and let $p:X_U\to X_K$ be the corresponding connected cover, that is the one satisfying $\pi_1(X_U)\cong U$.
	\end{hyp}
	
	\begin{definition}[Meridional inertia subgroups]
		For $g\in G$, define the \emph{meridional inertia subgroup} inside $U$ by
		\[ I_g(U):= U\cap g\langle m\rangle g^{-1}.\]
		Note that different elements $g,g'\in G$ can give the same subgroup $U\cap g\langle m\rangle g^{-1}$ (for instance, if $g$ and $g'$ differ by an element normalizing $\langle m\rangle$), but this redundancy is harmless for what follows.
	\end{definition}
	
	\begin{definition}[Ramification subgroup]
		The \emph{ramification subgroup} of $U$ is the normal closure in $U$ of all meridional inertia:
		\[M_U := \left\langle\!\left\langle I_g(U)\ :\ g\in G \right\rangle\!\right\rangle_U.\]
		Equivalently, $M_U$ is the smallest normal subgroup of $U$ containing every $U\cap g\langle m\rangle g^{-1}$.
	\end{definition}
	\begin{remark}
		The group $M_U$ was introduced, in different but equivalent terms, in \cite{PasiniClassifying, PasiniAmbClassSp}, where it was used to show the precise relationship between the unramified covering $p:X_U\to X_K$ and the corresponding ramified covering ramified over the knot. In geometric terms, $M_U$ is the subgroup of $U$ generated by the meridians of the connected components of the boundary of $X_U$. Each such connected component is a torus in its own right, and a covering sheet in the preimage of the boundary of $X_G$ through the covering map $p$.
	\end{remark}
	
	\section{The maximal meridionally unramified quotient}
	We illustrate the behavioral similarity between $M_U$ and number-theoretic ramification groups in three ways:
	\begin{enumerate}
		\item a universal property of $U/M_U$;
		\item a concrete description of the image of $M_U$ in \emph{every} finite quotient of $U$; and
		\item a discrete cohomological characterization of meridionally unramified $H^1$-classes.
	\end{enumerate}
	
	\begin{theorem}[Universal property of $U/M_U$]\label{thm:universal}
		Let $U\le G$ and $M_U$ be as above. Let $Q$ be any group and let $\varphi:U\to Q$ be a homomorphism such that for every $g\in G$ one has $\varphi(I_g(U))=\{1\}$ (that is, $\varphi$ annihilates every meridional inertia subgroup). Then there exists a unique $\overline{\varphi}:U/M_U\to Q$ with
		\[ 	\varphi = \overline{\varphi}\circ \pi.\]
	\end{theorem}
	
	\begin{proof}
		Let $S=\bigcup_{g\in G} I_g(U)\subseteq U$. By hypothesis, $\varphi(s)=1$ for all $s\in S$.
		Since $\varphi$ is a homomorphism, $\varphi$ annihilates every conjugate of $S$ inside $U$, hence $M_U=\langle\!\langle S\rangle\!\rangle_U\subseteq \ker(\varphi)$.
		
		By the first isomorphism theorem, any homomorphism with $M_U\subseteq\ker(\varphi)$ factors uniquely through the quotient $U/M_U$. Concretely, define $\overline{\varphi}(uM_U):=\varphi(u)$. This is well-defined because if $uM_U=u'M_U$ then $u^{-1}u'\in M_U\subseteq\ker(\varphi)$, so $\varphi(u)=\varphi(u')$. The homomorphism and uniqueness properties of $\overline\varphi$ are immediate.
	\end{proof}
	
	\begin{corollary}[Finite-quotient description of ramification]\label{cor:finitequot}
		Let $q:U\twoheadrightarrow F$ be a group epimorphism onto a finite group $F$. Then the image $q(M_U)\trianglelefteq F$ is the normal closure in $F$ of the images $q(I_g(U))$:
		\[q(M_U)=\left\langle\!\left\langle q(I_g(U))\ :\ g\in G\right\rangle\!\right\rangle_F.\]
		Equivalently, $F/q(M_U)$ is the largest quotient of $F$ in which the images of all $I_g(U)$ are trivial.
	\end{corollary}
	
	\begin{proof}
		Since $q$ is surjective, every conjugate in $F$ arises from a conjugate in $U$, so the image of the normal closure of any set $S$ is the normal closure of the image of $S$: in particular,
		\[ q(M_U)= \left\langle\!\left\langle q(I_g(U))\ :\ g\in G\right\rangle\!\right\rangle_F.\]
		The final sentence follows directly from Theorem~\ref{thm:universal}.
	\end{proof}
	
	\begin{remark}
		Corollary~\ref{cor:finitequot} is the finite-group analogue of ``killing inertia'' in class field theory: to enforce that every meridional inertia subgroup becomes trivial in a finite quotient, one must quotient by the normal closure of their images.
		
		From a purely group-theoretic point of view, the finiteness condition on $F$ plays no role and can be omitted altogether. However, for the purposes of this paper, the restriction to finite epimorphic images $F$ of $U$ is natural: it mirrors the number-theoretic situation, where global class field theory describes ramification in \textit{finite} (abelian) quotients of Galois groups.
		In particular, throughout Section~3 we systematically consider finite-index subgroups of $U$ and the corresponding finite covering spaces, as well as the induced finite quotients of the maximal meridionally unramified quotient $U/M_U$. In this context, $F$ always arises as a finite quotient $F \cong U/N$ with $[U:N]<\infty$, or as a finite quotient of $U/M_U$, and it is only these finite quotients that enter the subsequent profinite and cohomological arguments.
	\end{remark}
	
	For the cohomological characterization of meridionally unramified cohomology classes, we use the identification $H^1(H;\mathbb F_p)\cong \mathrm{Hom}(H,\mathbb F_p)$, valid for any discrete group $H$.
	
	\begin{definition}[Meridionally unramified $H^1$-classes]
		Let $p$ be a prime. A class $\alpha\in H^1(U;\mathbb F_p)$ is called \emph{meridionally unramified} if its restriction to each meridional inertia subgroup is trivial:
		\[\alpha|_{U\cap g\langle m\rangle g^{-1}} = 0 \quad\text{in } H^1(U\cap g\langle m\rangle g^{-1};\mathbb F_p) \qquad\forall g\in G.\]
	\end{definition}
	
	\begin{theorem}[Cohomological characterization: discrete case]\label{thm:H1-discrete}
		For any prime $p$, the inflation map induces an isomorphism
		\begin{equation}\label{eq:infl}
			H^1(U/M_U;\mathbb F_p) \;\xrightarrow{\ \cong\ }\; \left\{\alpha\in H^1(U;\mathbb F_p)\ :\ \alpha\text{ is meridionally unramified}\right\}, \quad \beta\mapsto\beta\circ\pi.
		\end{equation}
		Equivalently, a homomorphism $U\to \mathbb F_p$ factors through $M_U$ if and only if it vanishes on every inertia subgroup $U\cap g\langle m\rangle g^{-1}$.
	\end{theorem}
	
	\begin{proof}
		Identify $H^1(U;\mathbb F_p)$ with $\mathrm{Hom}(U,\mathbb F_p)$. Let $\pi:U\twoheadrightarrow U/M_U$ be the quotient map.
		
		\emph{Step 1:} \emph{The image of the inflation map consists of meridionally unramified classes.}
		In fact, for any $\beta\in \mathrm{Hom}(U/M_U,\mathbb F_p)$, the inflated map $\beta\circ\pi$ vanishes on $M_U$ by construction. Since $M_U$ is the normal closure of the union of all inertia subgroups, it follows that $\beta\circ\pi$ vanishes on each inertia subgroup.
		
		\emph{Step 2:} \emph{Any meridionally unramified class factors through $U/M_U$.}
		Let $\alpha\in \mathrm{Hom}(U,\mathbb F_p)$ be meridionally unramified, so that it vanishes on every inertia subgroup. Then $\alpha$ also vanishes on the normal closure of their union: if $s$ lies in some inertia subgroup then $\alpha(s)=0$; hence for any $u\in U$,
		$\alpha(u s u^{-1})=\alpha(u)+\alpha(s)-\alpha(u)=0$,
		that is, $\alpha$ vanishes on every conjugate of every inertia element; therefore $\alpha$ vanishes on the subgroup generated by those conjugates, i.e.\ on $M_U$. Thus $M_U\subseteq\ker(\alpha)$, so $\alpha$ factors through $U/M_U$: there exists $\beta\in \mathrm{Hom}(U/M_U,\mathbb F_p)$ such that $\alpha=\beta\circ\pi$.
		
		\emph{Step 3:} \emph{Injectivity.}
		Step 1 and Step 2 show the map in \eqref{eq:infl} is onto the unramified subset. Injectivity is immediate: if $\beta\circ\pi=0$ then $\beta=0$ because $\pi$ is surjective.
	\end{proof}
	
	\section{Profinite ramification}
	
	Let $\widehat{G}$ be the profinite completion of $G$ and likewise $\widehat{U}$ for $U$. Being the knot group of a tame knot, $G$ is finitely generated. Since $G$ is dense in $\widehat{G}$, the latter is topologically finitely generated. Then, by Nikolov-Segal's Theorem \cite{nikseg}, every finite-index subgroup of $\widehat{G}$ is open. In particular, since $U$ has finite index in $G$, $\widehat U$ identifies with the closure of $U$ in $\widehat G$ and is an \emph{open} subgroup of $\widehat G$. We write $\overline{H}^{\,\widehat U}$ for the closure of a subset $H\subseteq U$ inside $\widehat U$, and simply $\overline{H}$ for closure inside $\widehat{G}$ (these agree when $H\subseteq U$, since $\widehat U$ is open in $\widehat G$).

	\begin{lemma}[Closures of intersections with open subgroups]\label{lem:closure-intersection}
		Let $U\le G$ have finite index, and let $A\le G$ be any subgroup. Then inside $\widehat G$ one has
		\[ \overline{U\cap A} = \widehat U\cap \overline{A}.\]
	\end{lemma}
	
	\begin{proof}
		We view $G$ as a dense subgroup of $\widehat G$. Because $U$ has finite index in $G$, its closure $\widehat U$ is an open-and-closed subgroup of $\widehat G$, and $G\cap\widehat U=U$.
		
		The inclusion $\overline{U\cap A}\subseteq\widehat U\cap\overline A$ is immediate, because the right-hand side is a closed set including both $U$ and $A$. For the reverse inclusion, let $x\in\widehat U\cap\overline A$ and let $W$ be any open neighborhood of $x$ in $\widehat G$. Then $W\cap\widehat U$ is open in $\widehat G$ since $\widehat U$ is open. Since $x\in\overline A$, the set $W\cap\widehat U$ meets $A$; choose $a\in A\cap W\cap\widehat U$. Now $a\in \widehat U\cap G=U$, so $a\in U\cap A\cap W$. Thus every neighborhood of $x$ meets $U\cap A$, so $x\in\overline{U\cap A}$.
		
		This is also the standard closure-intersection identity for open subspaces, applied in the profinite topology; compare Ribes--Zalesskii \cite{RibesZalesskii}.
	\end{proof}
	
	\begin{proposition}[Closure of inertia subgroups]\label{thm:closure-inertia}
		For each $g\in G$ one has
		\[ \overline{\,U\cap g\langle m\rangle g^{-1}\,} = \widehat U \cap \overline{\,g\langle m\rangle g^{-1}\,} \subseteq \widehat G.\]
	\end{proposition}
	
	\begin{proof}
		Apply Lemma~\ref{lem:closure-intersection} with $A=g\langle m\rangle g^{-1}$.
	\end{proof}
	
	\begin{proposition}[Closure of ramification subgroup]\label{thm:closure-ramification}
		Let $M_U$ be the ramification subgroup of $U$. Then its closure in $\widehat U$ is the \emph{closed normal subgroup} generated by the closures of inertia:
		\[ \overline{M_U}^{\,\widehat U} = \overline{\left\langle\!\left\langle U\cap g\langle m\rangle g^{-1}:\ g\in G \right\rangle\!\right\rangle_U}^{\,\widehat U}
		=\left\langle\!\left\langle \overline{U\cap g\langle m\rangle g^{-1}}^{\,\widehat U}:\ g\in G \right\rangle\!\right\rangle_{\widehat U}^{\mathrm{cl}},\]
		where the superscript $\mathrm{cl}$ means ``take the closure'' after forming the normal subgroup.
	\end{proposition}
	
	\begin{proof}
		
		Let provisionally $C = \left\langle\!\left\langle \overline{U\cap g\langle m\rangle g^{-1}}^{\,\widehat U}:\ g\in G \right\rangle\!\right\rangle_{\widehat U}^{\mathrm{cl}}$. For the inclusion $\overline{M_U}^{\,\widehat U} \subseteq C$, it suffices to observe that, for any $g\in G$,
		$I_g(U) \subseteq \overline{I_g(U)}\subseteq C$; therefore $M_U$, being generated by the $I_g(U)$, is included in $C$; whence $\overline{M_U}\subseteq C$ because $C$ is closed.
		For the reverse inclusion, for any $g\in G$, $I_g(U)\subseteq M_U$, hence $\overline{I_g(U)}\subseteq \overline{M_U}$, hence $C\subseteq\overline{M_U}$ because the latter is a normal subgroup.\end{proof}

	Since $\langle m\rangle\cong \mathbb Z$, its profinite completion is $\widehat{\mathbb Z}$, and the closure $\overline{\langle m\rangle}^{\,\widehat G}$ is a procyclic subgroup of $\widehat{G}$ isomorphic to $\widehat{\mathbb Z}$.
	
	\begin{definition}[Profinite meridian subgroup]
		A \emph{profinite meridian subgroup} of $\widehat G$ means a subgroup of the form $\overline{g\langle m\rangle g^{-1}}^{\,\widehat G}$ for some $g\in G$ (and hence all its conjugates inside $\widehat G$).
	\end{definition}
	
	It is convenient to define the corresponding \emph{closed} inertia and ramification directly in profinite terms.
	
	\begin{definition}[Closed meridional inertia and closed ramification]
		For $x\in \widehat G$, define the \emph{closed inertia subgroup}
		\[\widehat I_x(\widehat U):=\widehat U\cap x\,\overline{\langle m\rangle}^{\,\widehat G}\,x^{-1}.\]
		Further, define the \emph{closed ramification subgroup} of $\widehat U$ as the smallest \emph{closed} normal subgroup of $\widehat{U}$ containing all $\widehat I_x(\widehat U)$:
		\[\widehat M_{\widehat U}:= \left\langle\!\left\langle \widehat I_x(\widehat U):x\in \widehat G \right\rangle\!\right\rangle_{\widehat U}^{\mathrm{cl}}.\]
	\end{definition}
	
	\begin{lemma}[Agreement with closure of the discrete ramification subgroup]\label{lem:profinite-agrees}
		Let $U\le G$ have finite index. Then \[\widehat M_{\widehat U}=\overline{M_U}^{\,\widehat U}.\]
	\end{lemma}
	
	\begin{proof}
		We prove mutual inclusion:
		\begin{enumerate}
			\item $\overline{M_U}^{\,\widehat U}\subseteq \widehat M_{\widehat U}$. 
			
			Pick an arbitrary $g\in G$. Then $U\cap g\langle m\rangle g^{-1}\subseteq \widehat{U}\cap g\overline{\langle m\rangle}^{\widehat{G}} g^{-1}$, which is closed because conjugation is continuous. Therefore $\overline{U\cap g\langle m\rangle g^{-1}}^{\widehat{U}}\subseteq \widehat{U}\cap g\overline{\langle m\rangle}^{\widehat{G}} g^{-1}$. Using Proposition \ref{thm:closure-ramification}, all the generators of $\overline{M_U}^{\widehat{U}}$ as a closed normal subgroup are also generators of $\widehat M_{\widehat U}$ in the same sense, so $\overline{M_U}^{\,\widehat U}\subseteq \widehat M_{\widehat U}$.

			\item $\widehat M_{\widehat U}\subseteq \overline{M_U}^{\,\widehat U}$. Fix $x\in \widehat G$. Because $G$ is dense in $\widehat G$, there exists a net $\{g_i\}_i\subseteq G$ converging to $x$ in $\widehat G$. By continuity of the conjugation map $(x,y)\mapsto xyx^{-1}$ in both arguments, for any $h\in \overline{\langle m\rangle}^{\,\widehat G}$, we have $g_i h g_i^{-1}\to x h x^{-1}$ and, if $\{h_j\}_j\subseteq{\langle m\rangle}$ is a net converging to $h$, $g_ih_jg_i^{-1}\to g_ih g_i^{-1}$. Thus every element of \(\widehat I_x(\widehat U)=\widehat U\cap x\,\overline{\langle m\rangle}^{\,\widehat G}x^{-1}\)
			is a limit of elements lying in the subgroups $\widehat U\cap g_i\,{\langle m\rangle}g_i^{-1}$.
			
			For each $i$ we have
			\[\widehat U\cap g_i\,\overline{\langle m\rangle}^{\,\widehat G}g_i^{-1}\subseteq \overline{U\cap g_i\langle m\rangle g_i^{-1}}^{\,\widehat U}\subseteq \overline{M_U}^{\,\widehat U}.\]
			Therefore any limit of elements from these subgroups also lies in $\overline{M_U}^{\,\widehat U}$, proving
			\({\widehat I_x(\widehat U)\subseteq \overline{M_U}^{\,\widehat U},}\)
			namely that all the generators of $\widehat M_{\widehat U}$ as a closed normal subgroup are also in $\overline{M_U}^{\widehat{U}}$. Thus $\widehat M_{\widehat U}\subseteq \overline{M_U}^{\,\widehat U}$.
		\end{enumerate}
	\end{proof}
	
	\section{Meridian-preserving profinite isomorphisms}
	
	In this section we prove a profinite-structure theorem of the following type: if one has an isomorphism between profinite completions of knot groups that \emph{respects the meridian (up to conjugacy)}, then it preserves all closed meridional inertia subgroups and hence the closed ramification subgroup.
	
	\begin{remark}
		In general, identifying ``the meridian'' intrinsically inside $\widehat G$ is subtle: many procyclic subgroups exist in profinite groups. Therefore our main theorem assumes that a given profinite isomorphism sends the conjugacy class of the (closure of the) meridian subgroup of one knot group to that of the other. This is the natural profinite analogue of ``preserving inertia'' in Galois theory.
	\end{remark}
	
	Now let $K'\subset S^3$ be another knot with group $G'=G_{K'}$ and meridian $m'\in G'$. Write $\widehat G'$ for the profinite completion.
	
	\begin{theorem}[Meridian-preserving profinite isomorphisms preserve closed inertia]\label{thm:profinite-inertia-rigidity}
		Assume there is a topological group isomorphism
		\[\Phi:\widehat G \xrightarrow{\ \cong\ } \widehat G',\] such that $\Phi$ sends the conjugacy class of the profinite meridian subgroup $\overline{\langle m\rangle}^{\,\widehat G}$ to the conjugacy class of $\overline{\langle m'\rangle}^{\,\widehat G'}$; equivalently, there exists $y\in \widehat G'$ with
		\[ \Phi\big(\overline{\langle m\rangle}^{\,\widehat G}\big)= y\,\overline{\langle m'\rangle}^{\,\widehat G'}\,y^{-1}.\]
		Let $\widehat U\le \widehat G$ be any open subgroup and set $\widehat U':=\Phi(\widehat U)$.
		Then for every $x\in \widehat G$ one has
		\[ \Phi\big(\widehat I_x(\widehat U)\big)=\widehat I_{\Phi(x)y}(\widehat U')\]
		(where inertia on the right is computed using $m'$). In particular, $\Phi$ induces a bijection between the families of closed inertia subgroups of $\widehat U$ and of $\widehat U'$.
	\end{theorem}
	
	\begin{proof}
		Fix an open subgroup $\widehat U\le \widehat G$, let $\widehat U'=\Phi(\widehat U)$, and fix $x\in \widehat G$. By definition,
		\[ \widehat I_x(\widehat U)=\widehat U\cap x\,\overline{\langle m\rangle}\,x^{-1}.\]
		Apply $\Phi$ to both sides. Because $\Phi$ is an isomorphism of topological groups, it preserves intersections and conjugation:
		\[ \Phi(A\cap B)=\Phi(A)\cap \Phi(B),\qquad \Phi(xAx^{-1})=\Phi(x)\Phi(A)\Phi(x)^{-1}.\]
		Hence
		\begin{align*}
			\Phi\big(\widehat I_x(\widehat U)\big) &= \Phi(\widehat U)\cap \Phi\big(x\,\overline{\langle m\rangle}\,x^{-1}\big) \\
			&= \widehat U'\cap \Phi(x)\,\Phi\big(\overline{\langle m\rangle}\big)\,\Phi(x)^{-1}.
		\end{align*}
		
		By hypothesis, $\Phi(\overline{\langle m\rangle})=y\,\overline{\langle m'\rangle}\,y^{-1}$ for some $y\in \widehat G'$. (Here $y$ is determined only up to right--multiplication by an element of $\overline{\langle m'\rangle}$, but this ambiguity does not affect the inertia subgroups.) Substituting in the previous equation:
		\begin{align*}
			\Phi\big(\widehat I_x(\widehat U)\big) &= \widehat U'\cap \Phi(x)\,y\,\overline{\langle m'\rangle}\,y^{-1}\,\Phi(x)^{-1} \\
			&= \widehat U'\cap (\Phi(x)y)\,\overline{\langle m'\rangle}\,(\Phi(x)y)^{-1},
		\end{align*}
		which is exactly $\widehat I_{\Phi(x)y}(\widehat U')$ by definition (with meridian $m'$). This proves the claimed equality.
		
		Finally, since $x$ ranges over all of $\widehat G$, so does $\Phi(x)y$ range over all of $\widehat G'$, hence the correspondence is bijective at the level of families.
	\end{proof}
	
	\begin{corollary}[Meridian-preserving profinite isomorphisms preserve closed ramification]\label{thm:profinite-ramification-rigidity}
	Under the hypotheses of Theorem~\ref{thm:profinite-inertia-rigidity}, for every open subgroup $\widehat U\le \widehat G$ with $\widehat U'=\Phi(\widehat U)$ one has
		\[ \Phi\big(\widehat M_{\widehat U}\big)=\widehat M_{\widehat U'}.\]
		Moreover, $\Phi$ induces a canonical isomorphism of quotients
		\[ \widehat U/\widehat M_{\widehat U}\ \cong\ \widehat U'/\widehat M_{\widehat U'}.\]
	\end{corollary}

\begin{proof}
	For every $x\in \widehat G$, Theorem~\ref{thm:profinite-inertia-rigidity} identifies $\Phi(\widehat I_x(\widehat U))$ as a closed inertia subgroup of $\widehat U'$. Therefore $\Phi(\widehat M_{\widehat U})$ is a closed normal subgroup of $\widehat U'$ containing all closed inertia subgroups of $\widehat U'$. By the definition of $\widehat M_{\widehat U'}$, we obtain $\widehat M_{\widehat U'} \subseteq \Phi(\widehat M_{\widehat U})$. The reverse inclusion follows applying Theorem~\ref{thm:profinite-inertia-rigidity} to $\Phi^{-1}$.
	Finally, since $\Phi$ restricts to an isomorphism $\widehat U\to \widehat U'$ sending $\widehat M_{\widehat U}$ to $\widehat M_{\widehat U'}$, it induces an isomorphism of the corresponding quotients.
\end{proof}

\begin{corollary}[Discrete version for finite-index subgroups]\label{cor:discrete-profinite-quotients}
	Let $U\le G$ be a subgroup of finite index and let $\widehat U$ be its closure in $\widehat G$. Under the hypotheses of Theorem~\ref{thm:profinite-ramification-rigidity}, writing $U'\le G'$ for the finite-index subgroup corresponding to $\widehat U'=\Phi(\widehat U)$ (under the standard correspondence between finite-index subgroups and open subgroups of the profinite completion), one obtains a canonical isomorphism
	\[ \widehat U/\overline{M_U}^{\,\widehat U}\ \cong\ \widehat U'/\overline{M_{U'}}^{\,\widehat U'}.\]
\end{corollary}

\begin{proof}
	By Lemma~\ref{lem:profinite-agrees}, $\widehat M_{\widehat U}=\overline{M_U}^{\,\widehat U}$ and similarly on the $G'$ side. Now apply Theorem~\ref{thm:profinite-ramification-rigidity} to identify the quotients.
\end{proof}

\begin{remark}[Why the meridian-preserving hypothesis is the right one]
	The conclusion cannot hold for an \emph{arbitrary} profinite isomorphism $\widehat G\cong \widehat G'$, because the definition of meridional inertia singles out a specific procyclic subgroup (the closure of $\langle m\rangle$). Theorem~\ref{thm:profinite-ramification-rigidity} should be read as: \emph{once a profinite isomorphism identifies the appropriate inertia object, it automatically preserves all inertia and ramification data built from it.}
\end{remark}

\section{Profinite cohomology and unramified classes}

This section is devoted to the profinite analogue of the cohomological characterisation of unramified cohomology classes treated in Theorem \ref{thm:H1-discrete}.
For any profinite group $\mathcal H$, there is a natural identification of the first continuous cohomology $H^1_{\mathrm{cts}}(\mathcal H;\mathbb F_p)$ with the set of continuous homomorphisms $\mathrm{Hom}_{\mathrm{cts}}(\mathcal H,\mathbb F_p)$.

\begin{theorem}[Cohomological characterization: profinite case]\label{thm:H1-profinite}
	Let $\widehat U$ be the profinite completion of $U$ and let $\overline{M_U}^{\,\widehat U}$ be the closure of $M_U$ inside $\widehat U$ (the closed ramification subgroup). Then for any prime $p$, inflation induces an isomorphism
	\[ H^1_{\mathrm{cts}}\!\big(\widehat U/\overline{M_U}^{\,\widehat U};\mathbb F_p\big) \;\xrightarrow{\ \cong\ }\; \left\{ \alpha\in H^1_{\mathrm{cts}}(\widehat U;\mathbb F_p)\ :\ \alpha \text{ vanishes on every closed inertia subgroup } \overline{U\cap g\langle m\rangle g^{-1}}^{\,\widehat U} \right\}.\]
\end{theorem}

\begin{remark}
	The preceding cohomological characterization is directly analogous to the role of inertia in local Galois theory. Let $K$ be a $p$-adic local field with absolute Galois group $G_K$ and inertia subgroup $I_K \trianglelefteq G_K$. Then $I_K$ plays the role of a distinguished inertia subgroup, and one has a short exact sequence
	\[1 \longrightarrow I_K \longrightarrow G_K \longrightarrow G_K/I_K \longrightarrow 1,\]
	where $G_K/I_K \cong \widehat{\mathbb{Z}}$ is the Galois group of the maximal unramified extension $K^{\mathrm{un}}/K$ (see, for example, \cite[Chap.~VII, \S 5]{NSW}). Accordingly, for any prime $p$ one has
	\[H^1_{\mathrm{cts}}(G_K/I_K;\mathbb{F}_p) \cong \left\{ \alpha \in H^1_{\mathrm{cts}}(G_K;\mathbb{F}_p) : \alpha|_{I_K} = 0 \right\},\]
	i.e.\ unramified cohomology classes are exactly those that vanish on inertia (this is the usual description of unramified classes via inflation–restriction;  \cite[Chap.~II, \S4 and Chap.~VII, \S2]{NSW}).
	
	In contrast with the knot-theoretic setting for a single finite-index subgroup, where each closed meridional inertia subgroup is procyclic, the inertia subgroup $I_K$ of a local Galois group -- and in particular its wild inertia -- is typically a large, non-finitely generated profinite group; see \cite[Chap.~VII, \S 5]{NSW} for the structure of $I_K$ and its higher ramification filtration.
	Nevertheless, the same formal mechanism applies: killing inertia yields the maximal unramified quotient and detects unramified cohomology classes.
\end{remark}

\begin{proof}
	This is the same argument as Theorem~\ref{thm:H1-discrete}, but in the category of profinite groups and continuous homomorphisms.
	
	Identify $H^1_{\mathrm{cts}}(\widehat U;\mathbb F_p)$ with $\mathrm{Hom}_{\mathrm{cts}}(\widehat U,\mathbb F_p)$. Let $\widehat\pi:\widehat U\twoheadrightarrow \widehat U/\overline{M_U}^{\,\widehat U}$.
	
	If $\beta$ is a continuous homomorphism out of the quotient, then $\beta\circ\widehat\pi$ kills $\overline{M_U}^{\,\widehat U}$, hence kills every closed inertia subgroup contained in it.
	
	Conversely, if $\alpha:\widehat U\to\mathbb F_p$ is continuous and kills every closed inertia subgroup, then it kills the closed normal subgroup generated by them; by Proposition~\ref{thm:closure-ramification}, that closed normal subgroup is exactly $\overline{M_U}^{\,\widehat U}$. Therefore $\overline{M_U}^{\,\widehat U}\subseteq \ker(\alpha)$, so $\alpha$ factors uniquely through the quotient. Injectivity is again immediate from surjectivity of $\widehat\pi$.
\end{proof}

\begin{remark}
The analogy with local Galois theory is formal rather than literal at the
level of the size of inertia subgroups. In the knot-theoretic setting considered here, for a fixed finite-index subgroup $U\leq G$, each closed meridional inertia subgroup \[ \overline{U\cap g\langle m\rangle g^{-1}}^{\,\widehat U} \] is a closed subgroup of the procyclic group
\[ \overline{g\langle m\rangle g^{-1}}^{\,\widehat G}, \] and hence is itself procyclic or trivial.

By contrast, in local Galois theory, the inertia subgroup may contain a very large wild part. If $k$ is a local field with residue characteristic $p$, inertia $T_k\trianglelefteq G_k$ contains the ramification group $V_k$, and $V_k$ is a free pro-$p$ group of countably infinite rank; see \cite[Chap.~VII, \S5]{NSW}. At the same time, \[ G_k/T_k\cong \widehat{\mathbb Z}\] is the Galois group of the maximal unramified extension of $k$. Thus the analogy used in this paper concerns the formal role of inertia: killing inertia produces an unramified quotient, and degree-one unramified cohomology classes are precisely those that vanish on inertia.\end{remark}

\end{document}